\documentclass[12pt]{amsart}
\usepackage{amssymb,latexsym}
\usepackage{enumerate}
\usepackage[margin=2in]{geometry}

\makeatletter
\@namedef{subjclassname@2010}{%
  \textup{2010} Mathematics Subject Classification}
\makeatother

\newtheorem{thm}{Theorem}[section]
\newtheorem{cor}[thm]{Corollary}
\newtheorem{prop}[thm]{Proposition}
\newtheorem{lem}[thm]{Lemma}

\theoremstyle{definition}

\newtheorem{rem}[thm]{Remark}

\numberwithin{equation}{section}

\newcommand\A[1]{A_{#1}}

\newcommand\R{\mathbb{R}}
\newcommand\N{\mathbb{N}}

\newcommand\Q{\mathbb{Q}}
\newcommand{\sr}[1]{\mathcal{A}^{[#1]}}
\newcommand{\srq}[4]{\mathcal{A}^{[#1]}_{#4}\left(#2,#3\right)}

\newcommand{\norma}[1]{\left\| #1 \right\| }
\newcommand{\abs}[1]{\left| #1 \right| }

\DeclareMathOperator{\supp}{supp}
\DeclareMathOperator{\dimH}{dim_H}

\begin{document}


\baselineskip=17pt



\title{Limit properties in a family of quasi-arithmetic means}

\author{Pawe{\l} Pasteczka}
\address{Institute of Mathematics\\University of Warsaw\\
02-097 Warszawa, Poland}
\email{ppasteczka@mimuw.edu.pl}


\begin{abstract}
It is known that the family of power means tends to maximum pointwise if we pass argument to infinity. We will give some necessary and sufficient condition for the family of quasi-arithmetic means generated by a functions satisfying certain smoothness conditions to have analogous property.
\end{abstract}

\maketitle

\section{Introduction}
There are several direction of exploration concerning means. Definitely the most popular are inequalities among different families of means. It could be seen in the by-now-classical monography \cite{Bullen}.

In the present paper we are going to discuss some property of the family of quasi-arithmetic means.
This family was introduced in the series of nearly simultaneous paper \cite{kolmogorov,nagumo,definetti} as a generalization of power means. Namely, for any continuous, strictly monotone function $f \colon U \rightarrow \R$ ($U$--an interval) one may define, for any vector of {\it entries} $a=(a_1,\ldots,a_r) \in U^r$ with {\it weights} $w=(w_1,\ldots,w_r)$, where $w_i>0$ and $\sum w_i=1$, a quasi arithmetic mean
$$\sr{f}(a,w):=f^{-1}(w_1f(a_1)+w_2f(a_2)+\cdots+w_rf(a_r)).$$

We will now pass into some limit approaches in a family of means. For a given sequence of means $(M_n)_{n \in \N}$ one could study (whenever exist) a pointwise limit $\lim_{n\rightarrow \infty} M_n$. In many families this limit is either one of element belonging to pertinent family, maximum or minimum. This trichotomy appears in many families of means. Perhaps the most known is the family of Power Means, but is also covers Gini means, Bonferroni means, Mixed means etc.

This phenomena, however, does not appear when it comes to quasi-arithmetic means. Some results concerning this family was proved by Kolesarova \cite{kolesarova}. In 2013 I provide in \cite{P2013} some results under additional smoothness condition (for a smooth and increasing family in a sense described later). Yet another result is implied by consideration of P\'ales \cite{Pa91} (cf. Lemma~\ref{lem:pales}), what was announced by him during our conversation.

We will discuss when the family generated by  $(f_n)_{n \in \N}$, $f_n\colon U \rightarrow \R$ tends to maximum pointwise. More precisely

$$\lim_{n \rightarrow \infty} \sr{f_n}(a,w) = \max(a)\textrm{ for any admissible }a\textrm{ and }w;
$$
hereafter such a family will be called {\it max-family}. Analogously we define min-family.
This definitions is valid for many different means, but very often some natural adaptation is required (e.g. omitting weights, restrict vector to fixed length).

\section{Auxiliary results}
To simplify many proofs among the present note we will restrict our consideration just to the two variable case. 
It will be denoted briefly by
$$\srq fxz\xi:=f^{-1}(\xi f(x)+(1-\xi)f(z)),\qquad x,z \in U,\,\xi \in (0,1).$$

We will request some equivalence-type lemma binding -[general] weighted quasi-arithmetic means, -quasi-arithmetic means of two variables and
-some approach established by Zs.~P\'ales \cite{Pa91}.
\begin{lem}
\label{lem:pales}
Let $U$ be an interval,  $(f_n)$ be a family of continuous, strictly monotone functions. Then the following conditions are equavalent
\begin{itemize}
\item[\textup{(i)}] $(f_n)$ is a max-family,
\item[\textup{(ii)}] $\srq {f_n}xz\xi \rightarrow \max(x,z)$ for $x,\,z \in U$ and $\xi \in (0,1)$,
\item[\textup{(iii)}] $\lim_{n \rightarrow \infty}\frac{f_n(x)-f_n(y)}{f_n(z)-f_n(y)} =0$ for all $x,\,y,\,z \in U$, $x<y<z$
\end{itemize}
\end{lem}

\begin{proof}
Implication $\textup{(i)} \Rightarrow \textup{(ii)}$ is trivial. To prove the opposite implication 
let us notice that $\sr{f}$ is symmetric. Moreover, for any $f$, $a \in U^r$ satisfying $a_1 \le a_2 \le \ldots \le a_r$ and admissible weights $w$, we have
$$\max(a) \ge \sr{f_n}(a,w) \ge \srq{f_n}{a_r}{a_1}{w_r}.$$
Passing $n \rightarrow \infty$ one gets
$$\lim_{n \rightarrow \infty} \sr{f_n}(a,w)=\max(a_1,a_r)=a_r=\max(a).$$

$\textup{(ii)} \Leftrightarrow \textup{(iii)}$ 
Let assume that each $f_n$ is increasing. Then 
$$\frac{f_n(x)-f_n(y)}{f_n(z)-f_n(y)} <0 \text{ for all }x,\,y,\,z \in U,\,x<y<z,\, n \in \N.$$
Assume $x,\,y,\,z \in U$, $x<y<z$ and $\xi \in (0,1)$. One simply gets
\begin{align*}
&\quad y < \srq{f_n}xz\xi,\\
&\iff f_n(y) < \xi f_n(x)+(1-\xi) f_n(z),\\
&\iff \frac{f_n(x)-f_n(y)}{f_n(z)-f_n(y)} > \tfrac{\xi-1}{\xi},\\
&\iff \frac{f_n(x)-f_n(y)}{f_n(z)-f_n(y)} \in \left(\tfrac{\xi-1}{\xi},0 \right).
\end{align*}
Passing $y \rightarrow z$ in the equivalence above provides the $(\Leftarrow)$ part, while passing $\xi \rightarrow 1$ provides the $(\Rightarrow)$ part.
\end{proof}

Now we are going to recall some variation of Mikusi\'nski's result \cite{Mikusinski}. 
\begin{lem}
Let $U$ be an interval, $f,\,g$ be a twice derivable functions with nowhere vanishing first derivative. If we denote $\A{f}:=f''/f'$, then the following conditions are equivalent:
\begin{itemize}
\item $\A{f}(x) \le \A{g}(x)$ for any $x \in U$,
\item $\sr{f}(a,w) \le \sr{g} (a,w)$ for any admissible $a$, $w$,
\item $\srq {f}xy\xi \le \srq gxy\xi$ for any admissible $x$, $y$ and $\xi$.
\end{itemize}
\end{lem}
Operator $\A{}$ is so central that, to make notion more compact, we will call a function {\it smooth} if it is twice derivable with nowhere vanishing first derivative -- in fact it is the weakest assumption making definition of $\A{}$ possible.

\begin{rem}
\label{rem:afin} The condition above has its own ,equal-type` version. Namely, the following conditions are equivalent:
\begin{itemize}
\item $\A{f}(x) = \A{g}(x)$ for any $x \in U$,
\item $\sr{f}(a,w) = \sr{g} (a,w)$ for any admissible $a$, $w$,
\item $\srq {f}xy\xi = \srq gxy\xi$ for any admissible $x$, $y$ and $\xi$,
\item $f=\alpha g+\beta$ for some $\alpha, \beta \in \R$, $\alpha \ne 0$.
\end{itemize}
 \end{rem}

For a family $(f_n)_{n\in \N}$, except of already introduced max-family, we will use attributes:
\begin{itemize}
\item {\it smooth} if $f_n$ smooth for any $n$,
\item {\it increasing} if $\sr{f_n}(a,w)\ge \sr{f_m}(a,w)$  for any $n \ge m$ and admissible $a$ and $w$ [under some additional assumptions, by Mikus\'nski's result, we obtain some equivalent definitions]
\item {\it lower bounded} if all functions are derivable and have nowhere vanishing derivative and there exists $C$ such that $f_n'(y)/f_n'(x) \ge e^{C(y-x)}$ for any $n \in \N$ and $x,\,y \in U$, $x<y$.
\end{itemize}

Analogously we define a dual definitions min-, decreasing and upper bounded family. In fact each result presented in this paper has its dual wording, which are omitted, but may be similarly established and proved. 
For a smooth functions it will be also handy to denote
\begin{equation}
X_\infty:=\{x \in U \colon \lim_{n \rightarrow \infty} \A{f_n}(x)=+\infty \}. \label{eq:Xinfty}
\end{equation}

Having this notation, let us recall some major result from \cite{P2013}:
\begin{prop}\label{prop:max_2013}
Let $U$ be a closed, bounded interval, $(f_n)_{n \in \N}$ be a smooth, increasing family defined on $U$.
\begin{itemize}
\item If $X_{\infty}=U$ then $(f_n)$ is a max-family.
\item If $(f_n)$ is a max-family then $X_{\infty}$ is a dense subset of $U$.
\end{itemize}
\end{prop}

\section{Main result}
It is a natural question, how to fulfilled a gap between necessary and sufficient condition
(cf. \cite[Open Problem]{P2013}).
The answer is fairly non-trivial. Namely, the fact if the family is a max-family cannot be
completely characterized by the properties of $X_\infty$. Some examples, counter-examples
as well as the strengthening of Proposition~\ref{prop:max_2013} will be given in section~\ref{sec:maxXinfty}.

At the moment we will give necessary and sufficient result for a family of derivable functions to be max.
Namely we are going to prove the following

\begin{thm}
\label{thm:main}
Let $U$ be an interval, $(f_n)$ be a family of derivable functions having nowhere vanishing first derivative, lower bounded family defined on a set $U$. Then
$$
\Big( (f_n) \textrm{ is a max-family} \Big)
\iff
\Bigg( \lim_{n \rightarrow \infty}\frac{f_n'(q)}{f_n'(p)} =\infty \text{ for all }p,\,q \in U,\,p<q \Bigg) 
.$$
\end{thm}

If a family $(f_n)$ consists of smooth functions then this theorem has an immediate
\begin{cor}
\label{cor:main}
Let $U$ be an interval, $(f_n)$ be a smooth, lower bounded family defined on a set $U$. Then
$$
\Big( (f_n) \textrm{ is a max-family} \Big)
\iff
\Bigg( \lim_{n \rightarrow \infty} \int_p^q \A{f_n}(x) dx =\infty \text{ for all }p,\,q \in U,\,p<q  \Bigg) 
.$$
\end{cor}

\section{Proof of Theorem~\ref{thm:main}}

Among all proof we will assume that the constant $C$ appearing in the definition of lower bound family is negative.

\subsection{Proof of $(\Rightarrow)$ implication}

Let us assume that there exists $x<z$ such that 
$$\liminf_{n \rightarrow \infty} \frac{f_n'(z)}{f_n'(x)} < \infty.$$ 
Then, there exists $\bar{H}>0$ and a subsequence $(n_1,n_2,\ldots)$ satisfying  
$$\frac{f_{n_k}'(z)}{f_{n_k}'(x)} < \bar{H},\quad k \in \N.$$ 
In particular, by lower bounded property, for any $k \in \N$ and $p,q \in[x,y]$,
\begin{align*}
\frac{f_{n_k}'(q)}{f_{n_k}'(p)} 
&= \frac{f_{n_k}'(z)}{f_{n_k}'(x)}\cdot\frac{f_{n_k}'(x)}{f_{n_k}'(p)}\cdot\frac{f_{n_k}'(q)}{f_{n_k}'(z)}
\le \frac{f_{n_k}'(z)}{f_{n_k}'(x)} e^{C(x-p)}e^{C(q-z)}
\le \bar{H}e^{2C(x-z)}
\end{align*}

Hence, for $H:=\bar{H}e^{2C(x-z)}$,
$$\frac{f_{n_k}'(q)}{f_{n_k}'(p)} < H,\quad k \in \N \textrm{ and }p,q \in[x,z].$$ 

Fix $y \in (x,z)$ and, by Remark~\ref{rem:afin}, assume $f_n'(y)=1$, $f_n(y)=0$, $n \in \N$. Whence
\begin{equation}
f_n(\tau)=\int_y^\tau \frac{f_n'(t)}{f_n'(y)} dt,\quad n \in \N,\,\tau \in U. \label{eq:ftau}
\end{equation}
one has $f_{n_k}(\tau) \le H\cdot(\tau-y)$ for $\tau \in (y,z)$.
In particular 
$$f_{n_k}(z) \le H\cdot(z-y)$$
Moreover for any $k \in \N$ we have the following implications:
\begin{align*}
f_{n_k}'(y)/f_{n_k}'(t) &\le H,\quad t \in (x,y),\\
f_{n_k}'(t)/f_{n_k}'(y)&\ge \tfrac1H,\quad t \in (x,y),\\
f_{n_k}'(t)&\ge \tfrac1H,\quad t \in (x,y),\\
\int_x^y f_{n_k}'(t) dt &\ge \tfrac{y-x}H,\\
-f_{n_k}(x) &\ge \tfrac{y-x}H,\\
f_{n_k}(x) &\le \tfrac{x-y}H.
\end{align*}

But $\tfrac{x-y}H<0$, therefore there exists $\xi>0$ such that
$$\xi f_{n_k}(x)+(1-\xi)f_{n_k}(z) <0,\quad k \in \N.$$
But, by~\eqref{eq:ftau}, $f_{n_k}$ is increasing and $f_{n_k}(y)=0$. Whence,
$$\srq {f_{n_k}}xz\xi<y,\quad k\in \N.$$

\subsection{$\Leftarrow$}
In the major part of the proof of this implication we will work toward the single function $f$ 
satisfying 
\begin{equation}
\frac{f'(y)}{f'(x)}\ge e^{C\cdot(y-x)} \text{ for some }C\textrm{ and all }x,\,y \in U,\,y>x. \label{eq:low_bd}
\end{equation}
Let us establish the following
\begin{lem}\label{lem:leftarrowtech}
Let $I$ be an interval, $f \colon I \rightarrow \R$ be a differentiable function with nowhere vanishing derivative satisfying \eqref{eq:low_bd} for some $C<0$.
Let us take $\xi \in (0,1)$ and $x,\,y,\,z\in I$ satisfying $x<y<z$. Then there exists
$\Phi=\Phi(\xi,C,\varepsilon,x,y)$ such that for any $\varepsilon \in (0,z-y)$
$$\frac{f'(z-\varepsilon)}{f'(y)} \ge \Phi \Rightarrow  \srq fxz\xi \ge y.$$
\end{lem}

It could be observe that the $(\Leftarrow)$ part of Theorem~\ref{thm:main} is implied by Lemma~\ref{lem:leftarrowtech}.
Indeed, by the definition, $\srq fxz\xi = \srq fzx{1-\xi}$ therefore let us assume, with no loss of generality, $z>x$. Moreover, by Lemma~\ref{lem:leftarrowtech}, there exists $n_\delta$ such that 
$$\frac{f_n'(z-\varepsilon)}{f_n'(y)} \ge \Phi(\xi,C,\varepsilon,x,z-\delta) \textrm{ for any }\delta>0 \textrm{ and }n>n_\delta.$$
Thus $\srq{f_n}xz\xi \ge z-\delta$ for any $\delta>0$ and $n>n_\delta$.
Whence 
$$\lim_{n \rightarrow \infty}\srq{f_n}xz\xi=z.$$

\subsection{Proof of Lemma~\ref{lem:leftarrowtech}}
By Remark~\ref{rem:afin}, let us assume, with no loss of generality,
$$f(\tau)=\int_y^\tau \frac{f'(t)}{f'(y)} dt,\,\tau \in U.$$

We will estate lower bound of $f(x)$ and, later, $f(z)$. By $f'(y)=1$, one obtains
$$f'(\kappa) \le e^{C(\kappa-y)},\,\kappa \in (x,y).$$
Whence,
$$f(y)-f(x) = \int_x^y f'(\kappa) d\kappa \le \int_x^y e^{C(\kappa-y)} d\kappa =\tfrac1C (1-e^{C(x-y)}).$$
Therefore, by $f(y)=0$,
$$f(x)\ge \tfrac1C (e^{C(x-y)}-1).$$

Let us calculate some lower bound of $f(z)$. Fix $\varepsilon \in (0,z-y)$. One has
\begin{align*}
f(z)
&=\int_y^z \frac{f'(t)}{f'(y)} dt \\
&\ge \int_{z-\varepsilon}^z \frac{f'(t)}{f'(y)} dt\\
&= \frac{f'(z-\varepsilon)}{f'(y)}\int_{z-\varepsilon}^z \frac{f'(t)}{f'(z-\varepsilon)}dt\\
&\ge \frac{f'(z-\varepsilon)}{f'(y)} \int_{z-\varepsilon}^z e^{C(t-z+\varepsilon)}dt\\
&=\frac{1}{C}  \cdot \frac{f'(z-\varepsilon)}{f'(y)} \left(e^{C\varepsilon}-1\right)
\end{align*}

We will prove that $\srq fxz\xi\ge y$ for sufficiently large $f'(z-\varepsilon)/f(y)$. Indeed, we have a series of $(\Leftarrow)$ implications:

\begin{align*} 
&\qquad y\le \srq{f}xz\xi\\
&\Leftrightarrow f(y)\le \xi f(x)+(1-\xi)f(z)\\
&\Leftarrow  0 \le \tfrac{\xi}C \cdot (e^{C(x-y)}-1) + \tfrac{1-\xi}{C} \cdot \frac{f'(z-\varepsilon)}{f'(y)} \left(e^{C\varepsilon}-1\right)\\
&\Leftrightarrow \frac{f'(z-\varepsilon)}{f'(y)} \ge \Phi \textrm{ for some }\Phi=\Phi(\xi,C,\varepsilon,x,y).
\end{align*}
In the previous, equivalent, calculation the reader should pay attention on the sign of each term. Lastly,
$$\frac{f'(z-\varepsilon)}{f'(y)} \ge \Phi \Rightarrow  \srq{f}xz\xi \ge y.$$

\section{\label{sec:maxXinfty} Relations between max-family and $X_\infty$ set}
At the moment we are heading toward possible strengthening of Proposition~\ref{prop:max_2013}.
We will present  a vary situations in a sequence of propositions [examples].
Let us denote by $\lambda,\,H^d,\,\dimH$ Lebesgue measure, $d$-dimensional Hausdorff measure and Hausdorff dimension, respectively.
Moreover used will be convention \eqref{eq:Xinfty}.

\begin{prop}
Let $U$ be an interval. Let $V \subset U$ such that there exists an open interval $W \subset U$ such that $\lambda(V \cap W)=0$. 
Then there exists a smooth, increasing family $(f_n)_{n \in \N}$, $f_n \colon U \rightarrow \R$, which is not max-family but $X_\infty \supset V$.
\end{prop}

\begin{proof}
With no loss of generality, let us assume $U=W$. Fix $\varepsilon>0$. 
We will construct an increasing family satisfying \textrm{(i)} $\norma{\A{f_n}}_{L_1(U)}<\varepsilon$ for any $n \in \N$
and \textrm{(ii)} $X_\infty \supset V$.
Then, by Corollary~\ref{cor:main}, this family is not max.

By the regularity of Lebesgue measure, there exists an open sets $V \Subset G_k \Subset H_k \subseteq G_{k-1}$, $G_0 \subset U$ and
$\lambda(H_k)<\frac{\varepsilon}{2^k}$.
By Tietze's theorem let us consider a family $s_k \colon U \rightarrow [0,1]$ of continuous functions
$$s_k(x)=\begin{cases} 
1 & x \in G_k, \\
0 & x \in U \backslash H_k, \\
\textrm{continuous prolongation} & x \in H_k \backslash G_k.
\end{cases}$$
Then $\norma{s_k}_{L_1(U)} < \frac{\varepsilon}{2^{k}}$ for any $n \in \N$. Moreover, let us take
$$\A{f_n}:=s_1+s_2+\cdots +s_n.$$
Therefore $\norma{\A{f_n}}_{L_1(U)}<\varepsilon$ for any $n \in \N$. Whence, by Corollary~\ref{cor:main}, $(f_n)_{n \in \N}$ is not a max-family. Nevertheless $\A{f_n}(x)=n$ for any $n \in \N$ and $x \in V$, so $X_\infty \supset V$.

\end{proof}

\begin{prop}
Let $U$ be an interval, $(f_n)_{n \in \N}$, $f_n \colon U \rightarrow \R$ be a smooth, increasing family $(f_n)_{n \in \N}$, $f_n \colon U \rightarrow \R$, if $\lambda(X_\infty \cap V)>0$ for any open subset $V \subset U$, then $(f_n)$ is a max-family.
\end{prop}

\begin{proof}
Assume $\A{f_n}> C$ for some $C<0$. Fix any $a,\,b \in U$, $a<b$. We have $\lambda(X_\infty \cap [a,b])>0$. 
For any $M>0$, one has
$$\bigcup_{n \in \N} \{x \in U \colon \A{f_n}(x)>M\} \supset X_{\infty}.$$
In particular, by the regularity of Lebesgue measure and monotonicity of $\A{f_n}$, there exists $n_M$ such that 
$$ \lambda \big([a,b] \cap X_\infty \backslash \{x \in U \colon \A{f_{n_M}}(x)>M\}\big)<1/M.$$
Equivalently,
$$ \lambda \big([a,b] \cap X_\infty \cap \{x \in U \colon \A{f_{n_M}}(x)>M\}\big)>\lambda \big([a,b] \cap X_\infty\big) - 1/M.$$

Now,
\begin{align*}
\int \limits_a^b \A{f_{n_M}} &\ge C \cdot (b-a) + M \cdot \lambda\big(\{x \in U \colon \A{f_{n_M}}(x)>M\}\big) \\
&\ge C \cdot (b-a) + M \cdot \lambda\big([a,b] \cap X_\infty \cap \{x \in U \colon \A{f_{n_M}}(x)>M\}\big) \\
&\ge C \cdot (b-a) + M \cdot \left( \lambda \big([a,b] \cap X_\infty\big) - 1/M \right) \\
&= C \cdot (b-a)-1+ M \cdot \lambda \big([a,b] \cap X_\infty\big)
\end{align*}
Upon taking a limit one gets
$$\lim_{M \rightarrow \infty} \int \limits_a^b \A{f_{n_M}} =\infty \textrm{ for any } a<b.$$
Therefore, by the the monotonicity property,
$$\lim_{n \rightarrow \infty} \int \limits_a^b \A{f_n} =\infty \textrm{ for any } a<b.$$
\end{proof}


\begin{prop}
Let $U$ be an interval. There exists an increasing max-family $(f_n)_{n \in \N}$, $f_n \colon U \rightarrow \R$ satisfying $\dimH(X_\infty)=0$.
\end{prop}

\begin{proof}

By Corollary~\ref{cor:main} it is equivalent to prove that $\int_U \A{f_n} \rightarrow \infty$ for any open interval $U$.
Let us enumerate a rational numbers in a set $U$: 
$$\Q \cap U=(q_1,q_2,\ldots).$$
Let
\begin{align*}
Q_{k}&=\bigcup_{i=1}^{k} B(q_i,\frac{1}{k^2 \cdot 2^i}),\\
{\hat Q}_{k}&=\bigcup_{i=1}^{k} B(q_i,\frac{2}{k^2 \cdot 2^i}).
\end{align*}
Then $Q_k$ is a finite sum of open intervals and $\abs{Q_k}\le \tfrac{1}{k^2}$, $\abs{\hat{Q_k}}\le \tfrac{2}{k^2}$.
By Tietze's theorem there exists functions $c_k \colon U \rightarrow [0,k^2]$, 
$$
c_k=
\begin{cases} k^2 & Q_{k}, \\
 0 & U \backslash \hat{Q_k}, \\
 \textrm{continuous prolongation} & \textrm{otherwise}.
\end{cases} 
$$
Let $\A{f_n}=c_1+\cdots+c_n$. Fix $x,\,y \in U$, $x<y$. By Corollary~\ref{cor:main} it is sufficient to prove that
$$\lim_{n \rightarrow \infty} \int_x^y \A{f_n}(u)du =\infty.$$

Let us take $q_i\in \Q \cap (x,y)$ and $k_0$ such that $B(q_i,\frac{1}{k_0^2 \cdot 2^i}) \subset (x,y)$.
Then, for $k>\max(i,k_0)=:k_0$, one obtains
$$\int_x^y c_k(u)du \ge \int_{B(q_i,1/(k^2 \cdot 2^i))} c_k(u) du = \int_{B(q_i,1/(k^2 \cdot 2^i))} k^2 du > 2^{1-i}. $$

Therefore, for $n> k_0$, 
$$
\int_x^y \A{f_n}(u)du
\ge \sum_{k=k_0+1}^n \int_x^y c_k(u)du
> (n-k_0)2^{1-i}.
$$
Whence
$$\lim_{n \rightarrow \infty} \int_x^y \A{f_n}(u)du = \infty.$$
So $(f_n)$ is a max-family. Moreover, note that $c_1(x)+\ldots +c_{n-1}(x) <n^3$ for any $n\in \N$ and $x \in U$. Thus
\begin{align*}
X_{\infty} &\subseteq \bigcap_{n=1}^{\infty} \bigcup_{k=n}^{\infty} \supp c_{k} \subseteq \bigcap_{n=1}^{\infty} \bigcup_{k=n}^{\infty} \hat{Q_k};\\
X_{\infty} &\subseteq \bigcup_{i=1}^{n} B\left(q_i,\frac2{n^2\cdot 2^i}\right) \cup \bigcup_{i=n+1}^{\infty} B\left(q_i,\frac2{i^2\cdot 2^i}\right),\quad n \in \N.
\end{align*}

Whence, for any $d>0$ and $n \in \N$, one gets
\begin{align*}
H^d(X_\infty) 
&\le \sum_{i=1}^{n} \left(\frac4{n^2\cdot 2^i} \right)^d + \sum_{i=n+1}^{\infty} \left(\frac4{i^2\cdot 2^i}\right)^d \\
&\le \frac{4^d}{n^{2d}} \sum_{i=1}^{\infty} \frac1{2^{id}}=\frac{4^d}{n^{2d} (1-2^{-d})}
\end{align*}
Passing $n \rightarrow \infty$ we get $H^d(X_\infty)=0$, $d>0$. So $\dimH(X_\infty)=0$.

\end{proof}
\begin{rem}
Let us notice that $X_\infty \supset \bigcap_{k=1}^{\infty} Q_k \supsetneq \Q$ because $\Q$ is not a $G_\delta$-set.
Moreover it is known that $X_\infty$ is a $G_\delta$-set for any max-family (cf. \cite[pp.204--205]{P2013}), whence $X_\infty$ cannot be a set of rational numbers.
\end{rem}

\end{document}